\documentclass[oneside,english]{amsart}
\usepackage[T1]{fontenc}
\usepackage{geometry}
\usepackage{color}
\usepackage{babel}
\usepackage{refstyle}
\usepackage{float}
\usepackage{units}
\usepackage{amstext}
\usepackage{amsthm}
\usepackage{amssymb}
\usepackage{graphicx}
\usepackage[all]{xy}
\usepackage[unicode=true,pdfusetitle,
 bookmarks=true,bookmarksnumbered=false,bookmarksopen=false,
 breaklinks=false,pdfborder={0 0 0},pdfborderstyle={},backref=false,colorlinks=true]
 {hyperref}
\hypersetup{
 citecolor=blue}

\makeatletter

\AtBeginDocument{\providecommand\subsecref[1]{\ref{subsec:#1}}}
\AtBeginDocument{\providecommand\lemref[1]{\ref{lem:#1}}}
\AtBeginDocument{\providecommand\figref[1]{\ref{fig:#1}}}
\AtBeginDocument{\providecommand\defref[1]{\ref{def:#1}}}
\AtBeginDocument{\providecommand\enuref[1]{\ref{enu:#1}}}
\AtBeginDocument{\providecommand\secref[1]{\ref{sec:#1}}}
\AtBeginDocument{\providecommand\thmref[1]{\ref{thm:#1}}}
\RS@ifundefined{subsecref}
  {\newref{subsec}{name = \RSsectxt}}
  {}
\RS@ifundefined{thmref}
  {\def\RSthmtxt{theorem~}\newref{thm}{name = \RSthmtxt}}
  {}
\RS@ifundefined{lemref}
  {\def\RSlemtxt{lemma~}\newref{lem}{name = \RSlemtxt}}
  {}

\numberwithin{equation}{section}
\numberwithin{figure}{section}
\theoremstyle{plain}
\newtheorem{thm}{\protect\theoremname}[section]
\theoremstyle{definition}
\newtheorem{defn}[thm]{\protect\definitionname}
\theoremstyle{plain}
\newtheorem{lem}[thm]{\protect\lemmaname}
\theoremstyle{definition}
\newtheorem{example}[thm]{\protect\examplename}
\theoremstyle{remark}
\newtheorem{rem}[thm]{\protect\remarkname}
\theoremstyle{plain}
\newtheorem{cor}[thm]{\protect\corollaryname}
\theoremstyle{definition}
\newtheorem{problem}[thm]{\protect\problemname}

\@ifundefined{date}{}{\date{}}
\makeatother

\providecommand{\corollaryname}{Corollary}
\providecommand{\definitionname}{Definition}
\providecommand{\examplename}{Example}
\providecommand{\lemmaname}{Lemma}
\providecommand{\problemname}{Problem}
\providecommand{\remarkname}{Remark}
\providecommand{\theoremname}{Theorem}

\begin{document}
\global\long\def\norm#1{\left\Vert #1\right\Vert }%
\global\long\def\QQ{\mathbb{Q}}%
\global\long\def\CC{\mathbb{C}}%
\global\long\def\ZZ{\mathbb{Z}}%
\global\long\def\NN{\mathbb{N}}%
\global\long\def\KK{\mathbb{K}}%
\global\long\def\RR{\mathbb{R}}%
\global\long\def\FF{\mathbb{F}}%
\global\long\def\oo{\mathcal{O}}%
\global\long\def\limfi#1#2{{\displaystyle \lim_{#1\to#2}}}%
\global\long\def\pp{\mathcal{P}}%
\global\long\def\qq{\mathcal{Q}}%
\global\long\def\da{\mathrm{da}}%
\global\long\def\dt{\mathrm{dt}}%
\global\long\def\ds{\mathrm{ds}}%
\global\long\def\dm{\mathrm{dm}}%
\global\long\def\flr#1{\left\lfloor #1\right\rfloor }%
\global\long\def\ceil#1{\left\lceil #1\right\rceil }%
\global\long\def\SL{\mathrm{SL}}%
\global\long\def\PGL{\mathrm{PGL}}%
\global\long\def\bb{\mathcal{B}}%
\global\long\def\GL{\mathrm{GL}}%
\global\long\def\PSL{\mathrm{PSL}}%

\title{Graphs with large girth and free groups}
\begin{abstract}
We use Margulis' construction together with lattice counting arguments
to build Cayley graphs on $\SL_{2}\left(\FF_{p}\right),\;p\to\infty$
which are $d$-regular graphs with girth $\geq\frac{2}{3}\frac{\ln\left(n\right)}{\ln\left(d-1\right)+\ln\left(C\right)}$
for some absolute constant $C$.
\end{abstract}

\author{Ofir David}
\email{eofirdavid@gmail.com}
\maketitle

\section{Introduction}

Starting with an empty graph on $n$ vertices, we can add $n-1$ edges
without creating any cycle, thus getting a tree, and every edge after
that must create a new cycle. However, if we choose the placements
of these new edges carefully, we can make sure that at least locally
our graph still looks like a tree, or equivalently we do not form
small cycles. We call the length of the shortest simple cycle the
\emph{girth of the graph }and we denote it by $\mathfrak{g}$.

Clearly, if we add too many edges than we must have small cycles.
To be more specific, suppose that we have a $d$-regular graph $\Gamma$
on $n$-vertices with an even girth $\mathfrak{g}=2m$ (though a similar
result holds for odd girth). In this case, any ball of radius $m-1$
in the graph is a tree, and since our graph is $d$-regular, this
tree has $1+d\frac{\left(d-1\right)^{m-1}-1}{d-2}\geq\left(d-1\right)^{m-1}$
vertices. It follows that $n\geq\left(d-1\right)^{m-1},$and by taking
the logarithm
\[
1+2\frac{\ln\left(n\right)}{\ln\left(d-1\right)}\geq\mathfrak{g}.
\]
Thus, we see that if we fix the number of vertices $n$, then we can't
have that both the girth $\mathfrak{g}$ is large and the number of
edges, which is controlled by $d$, is too large.

On the other hand, we can ask what is the largest girth possible for
a $d$-regular graph on $n$ vertices, and with this upper bound in
mind, define $C\left(n,d\right)$ to be the largest number such that
there exists a $d$-regular graph on $n$ vertices with
\[
\mathfrak{g}\geq C\left(n,d\right)\frac{\ln\left(n\right)}{\ln\left(d-1\right)}.
\]
The argument above shows that $C\left(n,d\right)\leq2$ (up to the
$+1$). One of the first lower bounds for $C\left(n,d\right)$ was
given by Erd{\"o}s and Saks in \cite{erdos_regulare_1963}, where
they used a counting method to show the existence of graphs with $\mathfrak{g}\geq\frac{\ln\left(n\right)}{\ln\left(d-1\right)}\left(1-o\left(n\right)\right)$.
However, probably the first explicit construction is by Margulis in
\cite{margulis_explicit_1982} where he showed that Cayley graphs
of $\SL_{2}\left(\FF_{p}\right)$ with the generators 
\[
S=\left\{ \left(\begin{array}{cc}
1 & \pm2\\
0 & 1
\end{array}\right),\left(\begin{array}{cc}
1 & 0\\
\pm2 & 1
\end{array}\right)\right\} 
\]
form a family of $4$-regular graphs with $\mathfrak{g}\ge\frac{2\cdot\ln\left(3\right)}{3\ln\left(1+\sqrt{2}\right)}\frac{\ln\left(n\right)}{\ln\left(4-1\right)}$
so that $C\left(n,4\right)\geq\frac{2\cdot\ln\left(3\right)}{3\ln\left(1+\sqrt{2}\right)}\sim0.831$,
and for general $d$ he shows that $C\left(n,d\right)\geq\frac{4}{9}$.
One of the interesting components of this proof is that $S$ can be
viewed as a subset of $\SL_{2}\left(\ZZ\right)$, and there it generates
a (finite index) free subgroup, so that its Cayley graph is the $4$-regular
tree (and similar arguments were used for general $d$). Thus, the
projections of groups $\left\langle S\right\rangle _{\SL_{2}\left(\ZZ\right)}\to\left\langle S\right\rangle _{\SL_{2}\left(\FF_{p}\right)}=\SL_{2}\left(\FF_{p}\right)$
induce covering maps from the $4$-regular tree to the finite graphs
in Margulis' construction. These graphs should be thought of as being
better and better approximations of the universal covering tree, and
what Margulis is showing is that this approximation is in a sense
``uniform'' (locally, the graphs look like trees) and ``fast''
(the girth is growing logarithmically in $n$). Margulis result was
later improved using similar ideas by Imrich in \cite{imrich_explicit_1984}
where he showed that $C\left(n,d\right)\geq0.48$.

One of the main benefits of working with Cayley graphs is that they
have many symmetries, and in particular they are vertex transitive.
It follows that in order to show that there are no small cycles, we
``only'' need to show that in a small neighborhood of a single vertex,
so in a sense we get the ``uniformity'' condition above automatically.
Furthermore, once we work with Cayley graphs, we have all the tools
from group theory in our disposal, which usually make things easier
and more interesting.

A second important construction leading to graphs with high girth
are the LPS graphs which were constructed by Lubotzky, Philips and
Sarnak (see \cite{lubotzky_ramanujan_1988}). These graphs are actually
Ramanujan graphs which are in a sense the best possible expander graphs,
and one of the properties of these graphs is that they have high girth.
As in Margulis' construction, the LPS graphs also arise from an algebraic
construction, this time from quaternion algebras, which in a sense
are very close to be matrix algebras, and here too there is a free
group which hides in the background. There are two types of LPS graphs,
on $\PSL_{2}\left(\FF_{p}\right)$ and $\PGL_{2}\left(\FF_{p}\right)$
respectively where the first satisfies $\mathfrak{g}\geq\frac{2}{3}\frac{\ln\left(n\right)}{\ln\left(d-1\right)}$
and the other $\mathfrak{g}\geq\frac{4}{3}\frac{\ln\left(n\right)}{\ln\left(d-1\right)}$.
In these graphs also one of the main component is to lift the graphs,
but instead of lifting to $\PGL_{2}\left(\ZZ\right)$, we lift them
to the $p$-adic group $\PGL_{2}\left(\ZZ_{p}\right)$. When we run
over the different primes $p$ (which can be put together inside the
adeles), we produce the different graphs. Thus, once again we have
a sort of a universal object, such that our family of graphs is just
increasing quotients of this object.

Other than these two construction, there are many more constructions,
see for example \cite{lazebnik_new_1995,lazebnik_explicit_1995},
though the LPS graphs still have the best result for large girth.
In this paper, we revisit Margulis' first construction and improve
it to get a family with $\mathfrak{g}\geq\frac{2}{3}\frac{\ln\left(n\right)}{\ln\left(d-1\right)}$.
While this doesn't improve upon the currently known results, it does
however uses a combination of ideas from combinatorics and group theory
which we find very interesting. Moreover, it also leads to questions
regarding lattice counting problems which seem natural and might suggest
generalization of this construction.
\begin{thm}
\label{thm:main}There exists an absolute constant $\tilde{C}$ and
$W_{R}\subseteq\SL_{2}\left(\ZZ\right)$ such that $W_{R}$ is a symmetric
basis for a free group, and the connected components $\Gamma_{p}$
of the identity of the Cayley graph $Cay\left(\SL_{2}\left(\FF_{p}\right),W_{R}\right)$
satisfy 
\[
\ceil{\frac{\mathfrak{g}}{2}}\geq\frac{1}{3}\frac{\ln\left(n_{p}\right)}{\ln\left|W_{R}\right|+\ln\left(\tilde{C}\right)},\quad n_{p}=\text{number of vertices in }\text{\ensuremath{\Gamma_{p}}}.
\]
\end{thm}

\subsection{Acknowledgment}

I would like to thank Ami Paz for first introducing me to this problem
of finding graphs with large girth.

The research leading to these results has received funding from the
European Research Council under the European Union Seventh Framework
Programme (FP/2007-2013) / ERC Grant Agreement n. 335989.

\newpage{}

\section{Constructing the graphs}

As mentioned before, our construction uses Cayley graphs of $\SL_{2}\left(\FF_{p}\right)$.
The basic argument for the lower bound on the girth is the same as
in Margulis' paper, while the difference will be in the choice of
generators which is related to lattice counting problems. With this
in mind, in section \subsecref{Cayley-graphs} we start by recalling
Margulis' proof, then in section \subsecref{alg_topology} we begin
to study how to construct free subgroups in $\SL_{2}\left(\ZZ\right)$
by considering them as fundamental groups of graphs. Finally, in \subsecref{Finding-good-generators}
we use these ideas to show how to construct sets of generators which
produce Cayley graphs with high girth.

\subsection{\label{subsec:Cayley-graphs}Cayley graphs and Margulis' proof}

We begin with the construction of Cayley graphs which will be our
examples of graphs with high girth.
\begin{defn}
Let $G$ be a group and $S\subseteq G$ a set. The Cayley graph $Cay\left(G,S\right)$
is the directed graph with $G$ as the set of vertices and the edges
$E=\left\{ g\to gs\;\mid\;\forall g\in G,\;s\in S\right\} $. We define
a labeling on the edges $L:E\to S$ by $L\left(g\to gs\right)=s$. 

Given an element $g\in G$ and a tuple $\bar{s}=\left(s_{1},...,s_{k}\right)\in S^{k}$
for some $k$, we define the path $P_{g,\bar{s}}$ to be
\[
P_{g,\bar{s}}=g\to gs_{1}\to gs_{1}s_{2}\to\cdots\to gs_{1}\cdots s_{k}.
\]
\end{defn}

The Cayley graphs are highly symmetric, and many of their combinatorial
properties can be formulated using the group $G$.
\begin{lem}
Let $G$ be a group and $S\subseteq G$.
\begin{enumerate}
\item The underlying undirected graph of $Cay\left(G,S\right)$ is connected
if and only if $S$ generates $G$.
\item The standard left action of $G$ on itself induces a left action of
$G$ on $Cay(G,S)$. In particular $Cay\left(G,S\right)$ is vertex
transitive.
\item A path $P_{g,\bar{s}},\;g\in G,\;\bar{s}\in S^{d}$ is nonbacktracking
if and only if $s_{i}s_{i+1}\neq e$ for each $i$, and it is a cycle
if and only if $s_{1}s_{2}\cdots s_{d}=e$. In particular, the distance
in the graph between $g$ and $gh$ is the word length of $h$ over
the elements in $S$ (which is infinite if $h\notin\left\langle S\right\rangle $).
\end{enumerate}
\end{lem}

\begin{proof}
These are all pretty easy, and we leave it as an exercise.
\end{proof}
\begin{example}
For $G=D_{6}$ the dihedral group with 6 elements and $S=\left\{ r,s\right\} $
where $r$ is the rotation and $s$ the reflection, we get the following
Cayley graph{\small{}
\[
\xymatrix{*++[o][F]{s}\ar[rrd]_{r}\ar@/_{.5pc}/[dd]_{s} &  &  &  & *++[o][F]{rs}\ar[llll]_{r}\ar@/_{.5pc}/[dd]_{s}\\
 &  & *++[o][F]{r^{2}s}\ar[rru]_{r}\ar@/_{.5pc}/[dd]\sb(0.35)s\\
*++[o][F]{e}\ar@/_{.5pc}/[uu]_{s}\ar[rrrr]\sp(0.3)r &  &  &  & *++[o][F]{r}\ar@/_{.5pc}/[uu]_{s}\ar[lld]^{r}\\
 &  & *++[o][F]{r^{2}}\ar@/_{.5pc}/[uu]\sb(0.65)s\ar[llu]^{r}
}
\]
}{\small\par}
\end{example}

If $S=S^{-1}$ is symmetric, then whenever $g\mapsto gs$ is an edge,
we also have the edge $gs\mapsto g=\left(gs\right)s^{-1}$. In this
case we will also think of $Cay\left(G,S\right)$ as an undirected
graph (after identifying these pair of edges), so we may talk about
the girth $\mathfrak{g}\left(Cay\left(G,S\right)\right)$. As with
the lemma above, this too can be formulated in the language of $G$
and $S$.
\begin{lem}
Let $G$ be a finite group and $S=S^{-1}\subseteq G$. Then the girth
$\mathfrak{g}\left(Cay\left(G,S\right)\right)$ is the smallest $d$
such that $\exists s_{i}\in S,\;i=1,...,d$ with $s_{i}s_{i+1}\neq e$
and $\prod_{1}^{d}s_{i}=e$.
\end{lem}

Things begin to be interesting when we have a homomorphism $\varphi:G\to H$
and for simplicity assume that $\varphi\mid_{S}$ is injective, so
we can identity $S$ with $\varphi\left(S\right)$. Such a homomorphism
will induce a graph homomorphism $\tilde{\varphi}:Cay\left(G,S\right)\to Cay\left(H,\varphi\left(S\right)\right)$,
so if $S$ is symmetric we immediately get that 
\[
\mathfrak{g}\left(Cay\left(G,S\right)\right)\geq\mathfrak{g}\left(Cay\left(H,\varphi\left(S\right)\right)\right).
\]

Margulis' idea was to fix $G=\SL_{2}\left(\ZZ\right)$, and to study
the standard congruence morphisms $\pi_{p}:\SL_{2}\left(\ZZ\right)\to\SL_{2}\left(\FF_{p}\right)$
(and note that $\pi_{p}\mid_{S}$ is injective for almost every $p$).
If we ever hope to get such a family with some lower bound $\mathfrak{g}\geq c\frac{\ln\left(n\right)}{\ln\left(d-1\right)}$,
then by the argument above the Cayley graph on $\SL_{2}\left(\ZZ\right)$
must be a tree. The connected components of a Cayley graph $Cay\left(G,S\right)$
are trees if and only if we cannot write $\prod_{1}^{d}s_{i}=e$ (a
cycle) without $s_{i}s_{i+1}=e$ for some $i$ (backtracking), so
that $S$ is a basis for a free group. Hence, a good place to look
for graphs with high girth, is with projections of a Cayley graph
of a free subgroup of $\SL_{2}\left(\ZZ\right)$.

Once the problem is in $Cay\left(\SL_{2}\left(\ZZ\right),S\right)$,
Margulis used the fact that we can use norms on $\SL_{2}\left(\ZZ\right)$.
So before we give the main idea of his proof, we need one definition
for norms.
\begin{defn}
Given a norm $\norm{\cdot}$ on $M_{2}\left(\RR\right)$, let $\eta\left(\norm{\cdot}\right)=\sup_{g,h\in\SL_{2}\left(\ZZ\right)}\frac{\norm g\norm h}{\norm{gh}}$. 
\end{defn}

\begin{example}
\begin{enumerate}
\item If the norm is multiplicative, for example the operator norm or the
$l_{2}$-norm, then $m\left(\norm{\cdot}\right)=1$.
\item For the max norm, it is easy to check that $\eta\left(\norm{\cdot}_{\infty}\right)=2$.
\end{enumerate}
\end{example}

\begin{lem}
\label{lem:girth}Let $F\leq\SL_{2}\left(\ZZ\right)$ be a free group
over a symmetric set $S=\left\{ s_{1}^{\pm1},...,s_{d}^{\pm1}\right\} $,
and let $p$ be a prime such that $\pi_{p}\mid_{S}$ is injective.
Let $\norm{\cdot}$ be any norm on $M_{2}\left(\RR\right)$ bigger
than $\norm{\cdot}_{\infty}$, and set $\eta=\eta\left(\norm{\cdot}\right)$
and $M={\displaystyle \max_{s\in S}}\norm s$. Then for $\mathfrak{g}=girth\left(Cay\left(\SL_{2}\left(p\right),S\right)\right)$
we get that $\left(\eta M\right)^{\ceil{\mathfrak{g}/2}}\geq\frac{\eta p}{2}$.
\end{lem}

\begin{proof}
Suppose that $\prod_{1}^{k}t_{i}\equiv_{p}I$ with $t_{i}\in S$ and
$t_{i}t_{i+1}\not\equiv_{p}I$, and in particular $t_{i}t_{i+1}\neq I$.
Since $F$ is free over $S$ this implies that $\prod_{1}^{k}t_{i}\neq I$.
Combinatorially speaking, a nonbacktracking cycle in $Cay\left(\SL_{2}\left(\FF_{p}\right),S\right)$
always lifts to a nonbacktracking path in $Cay\left(\SL_{2}\left(\ZZ\right),S\right)$
which is not a cycle.

Using the fact that $\SL_{2}\left(\ZZ\right)$ sits inside $M_{2}\left(\ZZ\right)$,
we get that that $\prod_{1}^{\flr{k/2}}t_{i}-\prod_{k}^{\flr{k/2}+1}t_{i}^{-1}\neq0$
over $\ZZ$, while it is 0 mod $p$, which implies that 
\[
2\cdot\max\left\{ \norm{\prod_{1}^{\flr{k/2}}t_{i}},\norm{\prod_{k}^{\flr{k/2}+1}t_{i}^{-1}}\right\} \geq\norm{\prod_{1}^{\flr{k/2}}t_{i}-\prod_{k}^{\flr{k/2}+1}t_{i}^{-1}}\geq\norm{\prod_{1}^{\flr{k/2}}t_{i}-\prod_{k}^{\flr{k/2}+1}t_{i}^{-1}}_{\infty}\geq p.
\]
Since $\norm{gh}\leq\eta\norm g\norm h$ for any $g,h\in\SL_{2}\left(\RR\right)$,
it follows that $\frac{\left(\eta M\right)^{\ceil{k/2}}}{\eta}\geq\frac{p}{2}$
which completes the proof.
\end{proof}
\newpage{}

If the elements of $S$ are distinct mod $p$, then $Cay\left(\SL_{2}\left(\FF_{p}\right),S\right)$
is $\left|S\right|$-regular. Given the lemma above, we want to find
a set $S=S^{-1}$ which is a basis of a free subgroup of $\SL_{2}\left(\ZZ\right)$
where $M={\displaystyle \max_{s\in S}}\norm s_{\infty}$ is as small
as possible, thus producing a graph with a high girth.

In Margulis' construction in \cite{margulis_explicit_1982}, he uses
\[
S=\left\{ \left(\begin{array}{cc}
1 & \pm2\\
0 & 1
\end{array}\right),\left(\begin{array}{cc}
1 & 0\\
\pm2 & 1
\end{array}\right)\right\} 
\]

and the operator norm, so that $\eta=1$ and $M=1+\sqrt{2}$. Additionally,
we have that $n=\left|\SL_{2}\left(\FF_{p}\right)\right|=p\left(p^{2}-1\right)$
, $d=4$ and $\SL_{2}\left(\FF_{p}\right)$ is generated by $S$.
Therefore by the lemma above we get that 
\[
\ceil{\mathfrak{g}/2}\geq\frac{\ln\left(\frac{p}{2}\right)}{\ln\left(1+\sqrt{2}\right)}=\frac{1}{3}\frac{\ln\left(p^{3}\right)-\ln\left(8\right)}{\ln\left(1+\sqrt{2}\right)}\geq\frac{\ln\left(3\right)}{3\ln\left(1+\sqrt{2}\right)}\frac{\ln\left(n\right)-\ln\left(8\right)}{\ln\left(4-1\right)}.
\]
Thus, this gives a construction with $\mathfrak{g}\geq0.831...\cdot\frac{\ln\left(n\right)}{\ln\left(d-1\right)}$.
\begin{rem}
In general, fixing the norm, if $d$ is big, then $M$ is going to
be big, so in the result of the \lemref{girth}, after taking the
logarithm, the constant $\eta$ will be negligible. In particular
this will be true if we let $d$ grow to infinity as well. In this
case we can simply think of the result (asymptotically) as 
\[
\mathfrak{g}\geq\frac{2}{3}\frac{\ln\left(n\right)}{\ln\left(M\right)}.
\]

In the rest of these notes we will only use the infinity norm, though
for small $d$, one might try to optimize the choice of the norm to
get better results.
\end{rem}

\subsection{\label{subsec:alg_topology}Free groups and graph covers}

Now that we have the basic idea of the proof and its relation to free
groups, we continue to construct free subgroups as fundamental groups
of graph covers.

Consider the following example of a labeled graph (defined below).

\begin{figure}[H]
\begin{centering}
\includegraphics[scale=0.5]{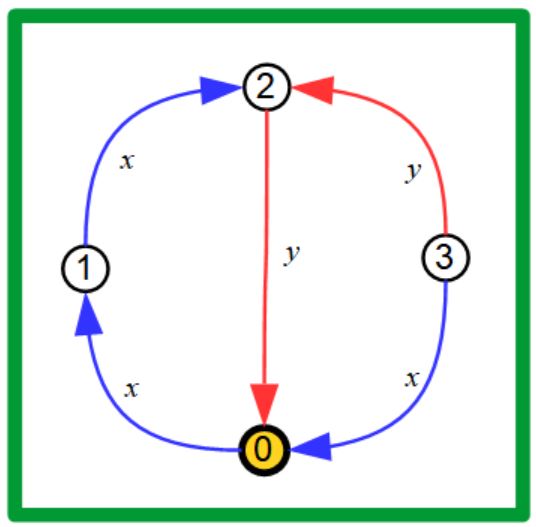}
\par\end{centering}
\caption{\label{fig:first_example} A labeled graph.}
\end{figure}
It has two ``main'' cycles corresponding to $x^{2}y$ on the left
and $x^{-1}y^{2}$ on the right, and every other cycle can be constructed
using these two cycles (up to homotopy, i.e. modulo backtracking).
Furthermore, the labeling allows us to think of these cycles as elements
in $F_{2}=\left\langle x,y\right\rangle $, so that the fundamental
group of the graph could be considered as the subgroup generated by
$x^{2}y$ and $x^{-1}y^{2}$. With this example in mind, we now give
the proper definitions to make this argument more precise.\\

One of the most basic results in algebraic topology is that the fundamental
group of a graph is always a free group. Let us recall some of the
details.
\begin{defn}[Cycle basis]
\label{def:Cycle_basis}Let $\Gamma$ be a connected undirected graph with a special vertex
$v\in V\left(\Gamma\right)$, and let $T\subseteq E\left(\Gamma\right)$
be a spanning tree. For each edge $e:u\to w$ let $C_{e}$ be the
simple cycle going from $v$ to $u$ on the unique path in the tree
$T$, then from $u$ to $w$ via $e$ and finally from $w$ to $v$
via $T$. We denote by $C\left(T\right)=\left\{ e\notin T\;\mid\;C_{e}\right\} $
this collection of cycles. 
\end{defn}

It is not hard to show that any cycle in a connected graph can be
written as a concatenation of cycles in $C\left(T\right)$ and their
inverses as elements in the fundamental group $\pi_{1}\left(\Gamma\right)$
(namely, we are allowed to remove backtracking). More over, it has
a unique such presentation which leads to the following:
\begin{cor}
\label{cor:cycle_basis}Let $\Gamma$ be a graph and $T$ a spanning
tree. Then $C\left(T\right)$ is a basis for $\pi_{1}\left(\Gamma\right)$
which is a free group on $\left|E\left(\Gamma\right)\right|-\left|V\left(\Gamma\right)\right|+1$
elements. 
\end{cor}

\begin{example}
In \figref{first_example} the edge touching the 0 vertex form a spanning
tree, and then $C_{\left(1,2\right)}=0\overset{x}{\to}1\overset{x}{\to}2\overset{y}{\to}0$
and $C_{\left(3,2\right)}=0\overset{x}{\leftarrow}3\overset{y}{\to}2\overset{y}{\to}0$,
so that eventually we will think of the fundamental group as generated
by $x^{2}y$ and $x^{-1}y^{2}$ per our intuition from the start of
this section.
\end{example}

In particular, the corollary above implies that the fundamental group
of the bouquet graph with a single vertex and $n$ self loops is the
free group $F_{n}$. We can label the edges by the corresponding basis
elements $x_{1},...,x_{n}$ in $F_{n}=\left\langle x_{1},...x_{n}\right\rangle $.
Since it is important in which direction we travel across the edge,
we will think of each edge as two directed edges labeled by $x_{i}$
and $x_{i}^{-1}$ depending on the image in the fundamental group.
For simplicity, we will keep only the edges with the $x_{i}$ labeling,
understanding that we can also travel in the opposite direction via
an $x_{i}^{-1}$ labeled edge. 

It is well known that a fundamental group of a covering space correspond
to a subgroup of the original space. Using the generalization of the
labeling above we can produce covering using the combinatorics of
labeled graphs.

For the rest of this section we fix a basis $x_{1},...,x_{n}$ of
the free group $F_{n}$. 
\begin{defn}
A labeled graph $\left(\Gamma,v\right)$ is a directed graph $\Gamma$
with a special vertex $v$, where the edges are labeled by $x_{1},...,x_{n}$
(see \figref{labeled_graphs}). A labeled graph morphism $\left(\Gamma_{1},v_{1}\right)\to\left(\Gamma_{2},v_{2}\right)$
between labeled graphs is a morphism of graphs $\Gamma_{1}\to\Gamma_{2}$
which sends $v_{1}$ to $v_{2}$ and preserves the labels on the edges. 
\end{defn}

We denote by $\Gamma_{F_{n}}$ the bouquet graph with the $x_{1},...,x_{n}$
labeling. Note that another way to define a labeling on a graph $\Gamma$
is a morphism of directed graphs $\varphi:\Gamma\to\Gamma_{F_{n}}$
where the labeling of an edge $e\in E\left(\Gamma\right)$ is defined
to be the labeling of $\varphi\left(e\right)$. In this way a labeled
graph morphism is just a map which defines a commuting diagram
\[
\xymatrix{\left(\Gamma_{1},v_{1}\right)\ar[r]\ar[rd] & \left(\Gamma_{2},v_{2}\right)\ar[d]\\
 & \Gamma_{F_{n}}.
}
\]
This labeling map $\varphi:\Gamma\to\Gamma_{F_{n}}$ induces a homomorphism
$\hat{\varphi}:\pi_{1}\left(\Gamma,v\right)\to\pi_{1}\left(\Gamma_{F_{n}}\right)=F_{n}$.
Since every path in $\Gamma$ is sent to a cycle in $\Gamma_{F_{n}}$,
we can extend this map to general paths in $\Gamma$.
\begin{defn}
Let $\left(\Gamma,v\right)$ be a graph with a labeling $\varphi:\Gamma\to\Gamma_{F_{n}}$.
Given a path $P$ in $\Gamma$ starting at $v$, define the labeling
$L\left(P\right)$ of the path to be the (cycle) element $\varphi\left(P\right)$
in the fundamental group $\pi_{1}\left(F_{n}\right)$. In other words,
this is just the element in $F_{n}$ created by the labels on the
path.
\end{defn}

In general, for a labeled graph $\varphi:\Gamma\to\Gamma_{F_{n}}$
the function $\hat{\varphi}$ is not injective. However, in the Stallings
graphs case, defined below, it is.
\begin{defn}
A Stallings graph is a labeled graph $\varphi:\left(\Gamma,v\right)\to\Gamma_{F_{n}}$
where $\varphi$ is locally injective, namely for every vertex $u\in V\left(\Gamma\right)$
and every $i=1,...,n$ there is at most one outgoing edge from $u$
and at most one ingoing edge into $u$ labeled by $x_{i}$. We call
the graph a covering graph if $\varphi$ is a local homeomorphism,
or equivalently every vertex has exactly one ingoing and one outgoing
labeled by $x_{i}$ for every $i$.
\end{defn}

\begin{rem}
Given a covering graph, we can remove every edge and vertex which
are not part of a simple cycle so as to not change the fundamental
group. The resulting graph will be a Stallings graph, and conversely,
every Stallings graph can be extended to a covering graph of $\Gamma_{F_{n}}$
without changing the fundamental domain.
\end{rem}

In the Stallings graph case, it is an exercise to show that $\hat{\varphi}$
is injective, and we may consider $\pi_{1}\left(\Gamma,v\right)$
as a subgroup of $F_{n}$. Moreover, we can use \defref{Cycle_basis}
to find a basis for $\pi_{1}\left(\Gamma,v\right)$ as a subgroup
of $F_{n}$.
\begin{example}
In \figref{labeled_graphs} below, in the left most graph, the path
\[
P:=v_{0}\overset{x}{\longrightarrow}v_{1}\overset{y}{\longrightarrow}v_{2}\overset{y}{\longrightarrow}v_{0}\overset{y}{\longleftarrow}v_{0}
\]
is labeled by $L\left(P\right)=xyyy^{-1}=xy$. Similarly, in the second
graph from the right the path 
\[
P:=v_{0}\overset{x}{\longrightarrow}v_{1}\overset{y}{\longrightarrow}v_{2}\overset{x}{\longleftarrow}v_{3}\overset{y}{\longleftarrow}v_{0}
\]
is labeled by $\pi_{1}\left(P\right)=xyx^{-1}y^{-1}=\left[x,y\right]$. 

The images $\hat{\varphi}\left(\Gamma,v\right)$ for the graphs in
this figure from left to right are $\left\langle xy,xy^{2},y\right\rangle =\left\langle x,y\right\rangle ,\;\left\langle y,xyx^{-1}\right\rangle $,
$\left\langle xyx^{-1}y^{-1}\right\rangle $ and $\left\langle x,y\right\rangle $.
Note that the fundamental group of the left most graph is free of
rank $3$ (there are 3 loops in the graph) while the image $\hat{\varphi}\left(\Gamma,v\right)=\left\langle x,y\right\rangle $
is generated by only two element, which in particular indicates that
it is not a Stallings graph.
\end{example}

\begin{figure}[H]
\begin{centering}
\includegraphics[scale=0.6]{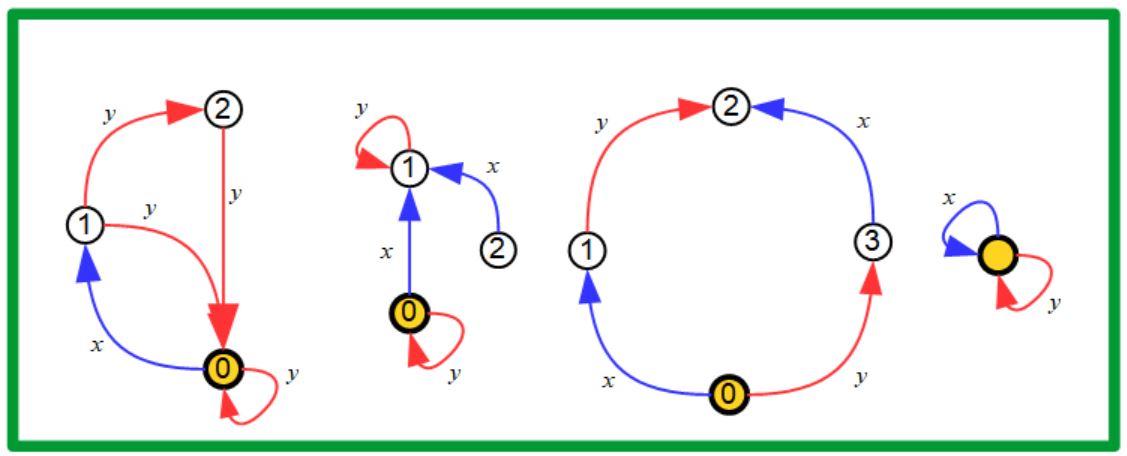}
\par\end{centering}
\caption{\label{fig:labeled_graphs}These are graphs labeled by $x,y$, where
$F_{2}=\left\langle x,y\right\rangle $ and the special vertices are
the yellow ones . The left most graph is not Stallings because it
has two $y$ labeled edges coming out of the same vertex. The rest
are Stallings graphs where the right most graph is $\Gamma_{F_{2}}$.}
\end{figure}

\subsection{\label{subsec:Finding-good-generators}Finding good generators}

Our final task is to construct graph covering, and to choose generators
which are small (so that $M$ in \lemref{girth} will be small). In
this section we will use the infinity norm, and just write $\norm{\cdot}$
instead of $\norm{\cdot}_{\infty}$. We will keep $S$ as the set
\[
S=\left\{ A,B\right\} ,\quad A=\left(\begin{array}{cc}
1 & 2\\
0 & 1
\end{array}\right),\;B=\left(\begin{array}{cc}
1 & 0\\
2 & 1
\end{array}\right)
\]
and we will look for free subgroups in $\left\langle S\right\rangle $.
As such, all of our graphs will be $S$ labeled (with our usual convention
for $A^{-1},B^{-1}$ labeled edges).

All of the elements in $S\cup S^{-1}$ are very ``similar'' and
in particular have the same norm. To make this even more precise,
we consider the following.
\begin{defn}
For $g\in\SL_{2}\left(\ZZ\right)$ define $\sigma\left(g\right)=\left(g^{-1}\right)^{T}$
and $\tau\left(g\right)=\left(\begin{array}{cc}
0 & 1\\
1 & 0
\end{array}\right)g\left(\begin{array}{cc}
0 & 1\\
1 & 0
\end{array}\right)$.
\end{defn}

\begin{lem}
The elements $\sigma$ and $\tau$ are commuting automorphisms of
$\SL_{2}\left(\RR\right)$ and each has order 2, so that $\left\langle \sigma,\tau\right\rangle =\left\{ e,\sigma,\tau,\sigma\tau\right\} $.
Moreover, $S\cup S^{-1}$ is a single orbit of $\left\langle \sigma,\tau\right\rangle $
and for any $g\in\SL_{2}\left(\RR\right)$ we have that $\norm g=\norm{\tau\left(g\right)}=\norm{\sigma\left(g\right)}$.
\end{lem}

\begin{proof}
Left as an exercise.
\end{proof}
We now use this symmetry to construct good free subgroups of $\left\langle S\right\rangle $.
The main idea will be to start with the 4-regular tree $Cay\left(\SL_{2}\left(\ZZ\right),S\right)$,
and for a fixed $R$ to look on the subgraph of the matrices $g$
with $\norm g\leq R$. We will then complete this graph to create
a covering graph of $\Gamma_{F_{2}}$. The bound on $\norm g$ will
imply that the max norm of our generators will be small, and then
we are left with the problem of counting how many elements satisfy
$\norm g\leq R$.
\begin{lem}
\label{lem:existance_free_groups}There exists an absolute constant
$C$ such that for any $R>0$ there exists a free subgroup $F=F_{R}\leq\left\langle S\right\rangle $
over a symmetric set $W=W_{R}$ satisfying:
\begin{enumerate}
\item \label{enu:tau_invariance}$W=\sigma\left(W\right)=\tau\left(W\right)$.
\item \label{enu:norm_bound}${\displaystyle \max_{w\in W}}\norm w\leq18R^{2}$.
\item \label{enu:lower_bound}$\left|W\right|\geq CR^{2}$.
\item \label{enu:distinct}The elements of $W$ are distinct mod $p$ for
$p>36\cdot R^{2}$.
\end{enumerate}
\end{lem}

\begin{proof}
Let $\Gamma''$ be the connected component of the identity in $Cay\left(\SL_{2}\left(\ZZ\right),S\right)$
which is a 4-regular tree with the natural $S$ labeling. Given $R>0$,
let $\Omega_{R}=\left\{ g\in\left\langle S\right\rangle :\norm g\leq R\right\} $
and set $\Gamma'=\left(V',E'\right)$ to be the smallest connected
subgraph of $\Gamma''$ which contains $\Omega_{R}$ (actually, it
can be shown that $V'=\Omega_{R}$). Since $\sigma$ and $\tau$ preserve
both $S^{\pm1}$ and the norm, it also acts on $\Gamma'$. The graph
$\Gamma'$ is an $S$ labeled tree and we want to complete it to be
a covering of $\Gamma_{F_{2}}$.

Let $g\in V'$ and $s\in S\cup S^{-1}$ such that $\left(g,gs\right)\notin E'$,
or equivalently $gs\notin V'$. Applying $\sigma$ we obtain that
$\sigma\left(g\right)\in V'$ while $\left(\sigma\left(g\right),\sigma\left(g\right)\sigma\left(s\right)\right)\notin E'$.
For each such pair $\left\{ \left(g,gs\right),\left(\sigma\left(g\right),\sigma\left(gs\right)\right)\right\} $
we add to $\Gamma'$ the vertex $v_{g,s}$, and two labeled edges
$g\overset{s}{\to}v_{g,s}\overset{s^{T}}{\to}g^{-T}$. In addition,
define $C_{g,s}$ to be the cycle defined by $e\rightsquigarrow g\overset{s}{\to}v_{g,s}\overset{s^{T}}{\to}\sigma\left(g\right)\rightsquigarrow e$
where the $\rightsquigarrow$ arrow is the unique path in the tree
$\Gamma'$. We let $\Gamma=\left(V,E\right)$ be the new graph after
doing this construction for each such pair, which construction is
a Stallings graph. Letting $W_{+}$ be the collection of cycles which
we just constructed, it is easily seen to be a cycle basis for $\Gamma$,
and therefore they generate $\pi_{1}\left(\Gamma\right)$, and we
need to show that $W_{R}:=W_{+}\cup W_{+}^{-1}$ satisfy the conditions
in this lemma. 

Note first that since by construction $\Gamma$ is a Stallings graph,
we can identify these cycles with elements in $\left\langle S\right\rangle $
via their labels, namely we identify $C_{g,s}\sim L\left(C_{g,s}\right)=\left(gs\right)\left(gs\right)^{T}$.
Since $\sigma$ and $\tau$ preserve the norm, by the definition of
these cycles we get that $W_{R}$ is invariant under $\sigma$ and
$\tau$ which is condition (\enuref{tau_invariance}).

Secondly, by definition for each cycle $C_{g,s}$ we have that $\norm{gs}\leq3\norm g\leq3R$,
so that 

\[
\norm{\left(gs\right)\left(gs\right)^{T}}\leq2\norm{gs}^{2}\leq18R^{2},
\]
which is part (\enuref{norm_bound}).

Next we want to find the size of $W_{R}$. For that, note first that
each edge in $E$ touches a vertex from $V'$ and each such vertex
in $V'$ has exactly two outgoing edges (labeled by $A$ and $B$
respectively), hence $\left|E\right|=2\left|V'\right|$. Recall that
$\Gamma'$ is a tree so it has $\left|E'\right|=\left|V'\right|-1$
edges, and any cycle $C_{g,s}$ uses exactly 2 edges in $E\backslash E'$.
Finally, two distinct cycles use different edges, so the number of
these cycles is 
\[
\frac{\left|E\right|-\left|E'\right|}{2}=\frac{2\left|V'\right|-\left(\left|V'\right|-1\right)}{2}=\frac{\left|V'\right|+1}{2},
\]
so we are left to count the number of vertices in $V'$ which is exactly
the number of matrices in $\left\langle S\right\rangle $ with norm
$\leq R$.

It is well known that $\SL_{2}\left(\ZZ\right)$ is a lattice in $\SL_{2}\left(\RR\right)$,
namely it is discrete and $\nicefrac{\SL_{2}\left(\RR\right)}{\SL_{2}\left(\ZZ\right)}$
has finite volume. Moreover, since $\left[\SL_{2}\left(\ZZ\right):\left\langle S\right\rangle \right]=12<\infty$,
the group $\left\langle S\right\rangle $ is also a lattice in $\SL_{2}\left(\RR\right)$.
By \cite{gorodnik_counting_2009} there is some absolute constant
$C$ such that $\left|V'\right|\geq CR^{2}$, which finishes part
(\enuref{lower_bound}). For completeness, we added an elementary
proof for this lower bound in the \secref{lattice_counting}.

Finally, assume that $w_{1},w_{2}\in W$ such that $w_{1}\equiv_{p}w_{2}$
where $2\cdot18R^{2}<p$. Since \\
$\norm{w_{1}-w_{2}}_{\infty}\leq2\cdot18R^{2}<p$, we must have that
$w_{1}=w_{2}$ which completes part (\enuref{distinct}) and the proof.
\end{proof}
\begin{rem}
As we shall see later, the condition that $W=\tau\left(W\right)$
is helpful when considering a possible extension of the results in
this section. If we ignore this condition, then there are many ways
to close the cycles in $\left(V',E'\right)$ which might have different
properties.
\end{rem}

Now that we have a way to construct a set of generators $W$ which
has small elements on the one hand (condition (\enuref{norm_bound})
above) and on the other hand $\left|W\right|$ is large (condition
(\enuref{lower_bound})), we can consider the family of graphs that
it creates when taken mod $p$, and use it to prove \thmref{main}.
\begin{proof}[Proof of \thmref{main}]
Let $W_{R},C$ be as in \lemref{existance_free_groups}. Note first
that $n\leq\left|\SL_{2}\left(\FF_{p}\right)\right|=p^{3}-p\leq p^{3},$and
more over by \lemref{existance_free_groups} we get that 
\[
M={\displaystyle \max_{w\in W}}\norm w\leq18R^{2}\leq18\frac{\left|W_{R}\right|}{C}.
\]
Assuming that $\pi_{p}$ is injective on $W_{R}$ (which is true for
$p>36R^{2}$) we get that $Cay\left(\SL_{2}\left(\FF_{p}\right),W_{R}\right)$
is $\left|W_{R}\right|$-regular graph on $n$ vertices. Applying
\lemref{girth} we obtain that 
\[
\ceil{\frac{\mathfrak{g}}{2}}\geq\frac{\ln\left(p\right)}{\ln\left(2M\right)}=\frac{1}{3}\frac{\ln\left(p^{3}\right)}{\ln\left(2M\right)}\geq\frac{1}{3}\frac{\ln\left(n\right)}{\ln\left(\frac{36}{C}\left|W_{R}\right|\right)}=\frac{1}{3}\frac{\ln\left(n\right)}{\ln\left(\left|W_{R}\right|\right)+\ln\left(\tilde{C}\right)},\;\tilde{C}=\frac{36}{C}.
\]
\end{proof}
As we mentioned before, since $\left|W_{R}\right|\to\infty$ as $R\to\infty$,
the term $\ln\left(\tilde{C}\right)$ is negligible, and also we can
move from from $\ln\left(d\right)=\ln\left|W_{R}\right|$ to $\ln\left(d-1\right)$
as in our discussion in the introduction. Hence, if we ignore this
``noise'' we get that $\mathfrak{g}\geq\frac{2}{3}\frac{\ln\left(n\right)}{\ln\left(d-1\right)}$.
While one can try to optimize the choice of norm and constants which
appear in the proof to get a tighter bound for some fixed $R$, our
interest is more in the asymptotic result, and the possible generalizations.

\section{Attempts at generalizations}

There are two directions at which one can try to generalize the results
from the previous section.

The first direction is to try and apply the same methods in $\SL_{n}\left(\ZZ\right)$
, $n\geq3$. Unlike the $n=2$ case, in $n\geq3$ there are no finite
index free groups in $\SL_{n}\left(\ZZ\right)$, so we cannot apply
the asymptotic growth result for lattices. Despite this, there are
many free subgroups in $\SL_{n}\left(\ZZ\right)$ which leads to the
following question about the lattice counting problem there.
\begin{problem}
Fix some $n\geq3$. Given $\varepsilon>0$, find a free subgroup $F_{\varepsilon}\leq\SL_{n}\left(\ZZ\right)$
and a constant $C_{\varepsilon}>0$ such that $\left|\left\{ g\in F_{\varepsilon}\ \mid\ \norm g_{\infty}\leq R\right\} \right|\geq C\mu_{n}\left(\left\{ g\in\SL_{n}\left(\RR\right)\ \mid\ \norm g_{\infty}\leq R\right\} \right)^{1-\varepsilon}$
where $\mu_{n}$ is the Haar measure of $\SL_{n}\left(\RR\right)$.
\end{problem}

\begin{rem}
Choosing two elements $g,h\in\SL_{n}\left(\ZZ\right)$ in random,
it is well known that with high probability they generate a free subgroup.
Furthermore, if $S$ is any symmetric set of generators of a free
group $F$, we obtain that $\norm{\prod_{1}^{k}s_{i}}_{\infty}\leq\left(n\cdot{\displaystyle \max_{1\leq i\leq k}}\norm{s_{i}}_{\infty}\right)^{k}$
for $s_{i}\in S$, so that 
\[
\left|\left\{ g\in F\ \mid\ \norm g_{\infty}\leq R\right\} \right|\gtrsim R^{\frac{\ln\left|S\right|}{\ln\left(n\right)+\ln\left(\max_{s\in S}\norm s_{\infty}\right)}}.
\]
In other words, the growth rate of any free group is at least polynomial
in $R$, and the problem is to find the best power of $R$ attainable.
Here, too, we can play with the choice of norm to get better exponents.
\end{rem}

The second problem with $n\geq3$, is that $\norm{g^{-1}}_{\infty}$
in general doesn't equal $\norm g_{\infty}$, and the trivial upper
bound is $\norm g_{\infty}^{-1}\leq\norm g_{\infty}^{n-1}\cdot\left(n-1\right)!$
using Cramer's rule, so a proper generalization should probably be
for the intersection $\norm g_{\infty}\leq R$ and $\norm{g^{-1}}_{\infty}\leq R$.\\

The second generalization is to use $\GL_{2}\left(\ZZ\right)$ and
$\PGL_{2}\left(\FF_{p}\right)$ instead of $\SL_{2}\left(\ZZ\right)$
and $\SL_{2}\left(\FF_{p}\right)$.

Let $W\subseteq\SL_{2}\left(\ZZ\right)$ be a symmetric set which
generates a free group and in addition assume that $\tau\left(W\right)=W$.
Letting $G\left(p\right)=\left\{ g\in\GL_{2}\left(\FF_{p}\right)\ \mid\ \det\left(p\right)=\pm1\right\} $,
we can consider the graphs $Cay\left(G\left(p\right),W\left(\begin{array}{cc}
0 & 1\\
1 & 0
\end{array}\right)\right)$. Since $\left[G\left(p\right):\SL_{2}\left(\FF_{p}\right)\right]=2$
and $W\left(\begin{array}{cc}
0 & 1\\
1 & 0
\end{array}\right)\subseteq G\left(p\right)-\SL_{2}\left(\FF_{p}\right)$, these graphs are bipartite. Moreover, a cycle in this graph corresponds
to the relation
\[
\prod_{1}^{2d}\left(w_{i}\left(\begin{array}{cc}
0 & 1\\
1 & 0
\end{array}\right)\right)=\prod_{1}^{d}\left(w_{2i-1}\left[\left(\begin{array}{cc}
0 & 1\\
1 & 0
\end{array}\right)w_{2i}\left(\begin{array}{cc}
0 & 1\\
1 & 0
\end{array}\right)\right]\right)=\prod_{1}^{d}\left(w_{2i-1}\tau\left(w_{2i}\right)\right)\equiv_{p}I
\]

The assumption that $W=\tau\left(W\right)$ implies that $girth\left(Cay\left(G\left(p\right),W\left(\begin{array}{cc}
0 & 1\\
1 & 0
\end{array}\right)\right)\right)=even-girth\left(Cay\left(\SL_{2}\left(p\right),W\right)\right)$, namely the length of the smallest even cycle in $Cay\left(\SL_{2}\left(p\right),W\right)$.
While clearly we have that $even-girth\left(Cay\left(\SL_{2}\left(p\right),W\right)\right)\leq2\cdot girth\left(Cay\left(\SL_{2}\left(p\right),W\right)\right)$,
if we can show that this bound is almost tight, we would obtain a
better family of graphs with high girth. In particular, if we can
show equality then the parameters $n,d,\mathfrak{g}$ of these graphs
will satisfy (asymptotically) $\mathfrak{g}\geq\frac{4}{3}\frac{\ln\left(n\right)}{\ln\left(d-1\right)}$.
This type of improvement is exactly what happens in the LPS graphs
in \cite{lubotzky_ramanujan_1988}.

\newpage{}

\appendix

\section{\label{sec:lattice_counting}Counting matrices from $\protect\SL_{2}\left(\protect\ZZ\right)$}

The aim of this appendix is to give a lower bound for the number of
$\SL_{2}\left(\ZZ\right)$ matrices in increasing balls, namely $\left|B_{R}\cap\SL_{2}\left(\ZZ\right)\right|$
where $B_{R}=\left\{ M\in M_{2}\left(\RR\right)\;\mid\;\norm M_{\infty}\leq R\right\} $.
This type of lattice counting problems are quite common in the literature,
and usually the idea is that since $\SL_{2}\left(\ZZ\right)$ is a
lattice in $\SL_{2}\left(\RR\right)$, then the number of lattice
points in increasing balls in $\SL_{2}\left(\RR\right)$ behaves like
the growth of the volumes of these balls (or the simpler case, which
is Gauss circle problem: the number of integer points inside growing
circle grows like $\pi r^{2}$ - the area of the circles). Note however
that the ball is usually defined using an invariant metric while we
use the infinity metric which is not $\SL_{2}\left(\RR\right)$-invariant.
This specific counting problem can still be solved using similar tools
(see for example \cite{gorodnik_counting_2009}), but for the reader's
convenience we added an elementary proof for the lower bound of $\SL_{2}\left(\ZZ\right)$-points
in increasing balls needed in this paper.\\

Recall that $\left(n,m\right)\in\ZZ^{2}$ is called primitive if $gcd\left(n,m\right)=1$.
We begin with the observation that such $\left(n,m\right)$ can be
completed to {\footnotesize{}$\left(\begin{array}{cc}
n & a\\
m & b
\end{array}\right)\in\SL_{2}\left(\ZZ\right)$} if and only if it is primitive. Clearly, if it is part of a matrix
in $\SL_{2}\left(\ZZ\right)$, then it is primitive. On the other
hand if $\left(n,m\right)$ is primitive, we can find $a,b\in\ZZ$
such that $nb-ma=1$ so that {\footnotesize{}$\left(\begin{array}{cc}
n & a\\
m & b
\end{array}\right)\in\SL_{2}\left(\ZZ\right)$}. In other words, the set of primitive vectors is exactly the orbit
$\SL_{2}\left(\ZZ\right)e_{1}$ where $e_{1},e_{2}$ is the standard
basis for $\RR^{2}$.

The second observation is that once we find such an $\left(a,b\right)$
any other completion to a basis is of the form $\left(a,b\right)+k\left(n,m\right)$
for some $k\in\ZZ$. One can see check this directly or note that
if $\left(n,m\right)^{T}=ge_{1}$ for some $g\in\SL_{2}\left(\ZZ\right)$,
then $\left(a,b\right)^{T}=ge_{2}$. Any other $g'\in\SL_{2}\left(\ZZ\right)$
which satisfies $\left(n,m\right)^{T}=g'e_{1}$ must be of the form
$gh$ with $he_{1}=e_{1}$, namely{\footnotesize{} $h=\left(\begin{array}{cc}
1 & k\\
0 & 1
\end{array}\right)$}, and therefore 
\[
g'e_{2}=g\left(ke_{1}+e_{2}\right)=\left(a,b\right)^{T}+k\left(m,n\right)^{T}.
\]

Let us interpret these observations geometrically. Suppose that we
have a primitive vector $\left(n,m\right)\in\ZZ^{2}$, and let $L_{\left(n,m\right)}=\left\{ \left(x,y\right)\in\RR^{2}\;\mid\;\left(x,y\right)\cdot\left(-m,n\right)=1\right\} $
be a line. The two observation tell us that this line contain infinitely
many integral points which complete $\left(n,m\right)$ to a matrix
in $\SL_{2}\left(\ZZ\right)$, and we can move from one solution to
the other by adding $\left(n,m\right)$. 

\begin{figure}[H]
\begin{centering}
\includegraphics[scale=0.4]{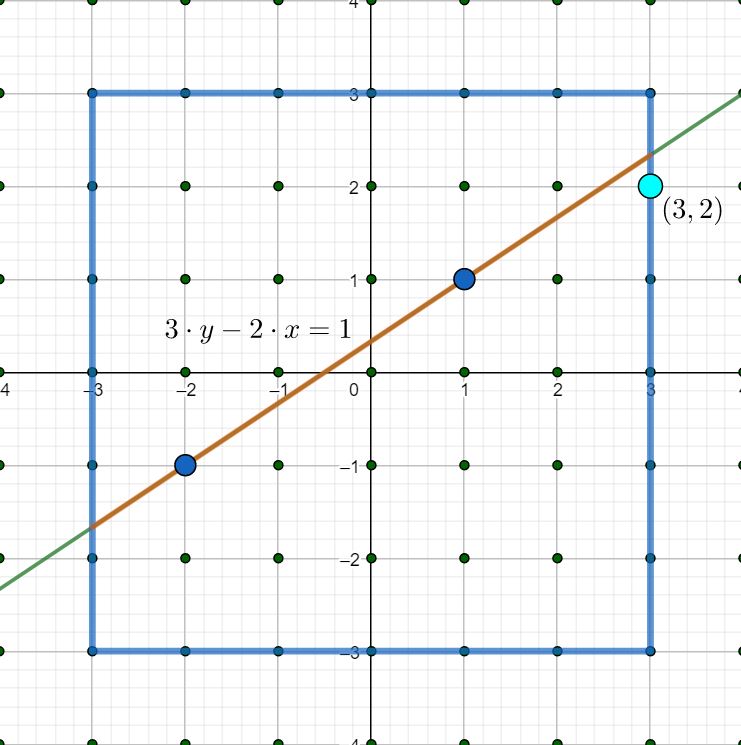}
\par\end{centering}
\caption{\label{fig:primitive}(Geogebra \cite{hohenwarter_geogebra:_2002})
The point $\left(3,2\right)$ and the line corresponding to completion
of $\left(3,2\right)$ to a matrix in $\protect\SL_{2}\left(\protect\RR\right)$.}
\end{figure}
As can be seen in the image, the intersection of the line $L_{\left(3,2\right)}$
with the square of radius $\norm{\left(3,2\right)}_{\infty}$ is a
translation of the segment $\left[-\left(3,2\right),\left(3,2\right)\right]$.
Since we can move from one integral point on this line to the next
by adding $\left(3,2\right)$, this intersection must contain two
integral points, which have infinity norm at most $\norm{\left(3,2\right)}_{\infty}$.
This idea allows us to prove the following.
\begin{defn}
Define 
\[
prim\left(R\right)=\left\{ \left(n,m\right)\in\NN^{2}\;\mid\;gcd\left(n,m\right)=1,\;\left|n\right|,\left|m\right|\leq R\right\} .
\]
\end{defn}

\begin{lem}
For every $R>0$ we have that $\left|B_{R}\cap\SL_{2}\left(\ZZ\right)\right|\geq prim\left(R\right)$.
\end{lem}

\begin{proof}
The main idea already appears in the sketch above, and we only need
to show that the image above is what really happens for any primitive
vector $\left(n,m\right)$. By switching $n$ and $m$ and multiplying
by $-1$ if need, we may assume that $0\leq m\leq n$. In this case
our line is $yn-xm=1$ so that {\footnotesize{}$\left(\begin{array}{cc}
n & x\\
m & y
\end{array}\right)\in\SL_{2}\left(\ZZ\right)$}. Because $n\neq0$, we can write it as $y=\frac{1+xm}{n}$. In particular,
the intersection of this line with $x=n$ is at $\left(n,\frac{1}{n}+m\right)$.
If $n>m$, then $m<m+\frac{1}{n}<n$ we the result is as in \figref{primitive}.
If $n=m$, then we must have that $n=m=1$, but then we can complete
$\left(1,1\right)$ with $\left(1,0\right)$, which again have norm
$\norm{\left(1,0\right)}_{\infty}=\norm{\left(1,1\right)}_{\infty}$.
In any way, we showed that if $\left(n,m\right)$ is primitive, we
can complete it to a matrix in $\SL_{2}\left(\ZZ\right)$ with a vector
$\left(a,b\right)$ such that $\norm{\left(a,b\right)}_{\infty}\leq\norm{\left(n,m\right)}_{\infty}$,
and this map from a primitive vector to $\SL_{2}\left(\ZZ\right)$
implies that $prim\left(R\right)\leq\left|B_{R}\cap\SL_{2}\left(\ZZ\right)\right|$.
\end{proof}
\begin{rem}
The idea for the lemma above can be also used to give an upper bound
$\left|B_{R}\cap\SL_{2}\left(\ZZ\right)\right|\leq C\cdot prim\left(R\right)$
for some $C>0$. However, this direction is a little bit more involved,
since while many of the primitive vectors $v$ of length $\norm v\leq R$
have very few completion to matrices with vectors of length $\leq R$,
if $\norm v$ is very small it can have many completions. For example,
we can complete $\left(1,0\right)$, with all the vectors of the form
$\left(k,1\right)$ with $k\leq R$ - which grows linearly in $R$.
As we do not require the upper bound for this paper, we leave it as
an exercise to the interested reader.
\end{rem}

We are now left with the problem of counting primitive vectors in
an increasing balls. This is a well known result which can be done
elementarily using the inclusion exclusion principle.
\begin{lem}
\label{lem:primitive}We have $\limfi N{\infty}\frac{\left|prim\left(N\right)\right|}{N^{2}}=4\cdot\zeta\left(2\right)^{-1}=4\cdot\frac{6}{\pi^{2}}$.
\end{lem}

\begin{proof}
The only primitive vector on the $x$ and $y$ axis are $\pm e_{1},\pm e_{2}$
and we have symmetry between the four quarters of the plane, so it
is enough to look on primitive vectors with positive coordinates.

Fix $N$ and for $P\in\NN$ set $U_{P}=\left[1,...,N\right]^{2}\cap P\ZZ^{2}$.
Then we want to find the size of $prim^{+}\left(N\right):=U_{1}\backslash\bigcup_{p}U_{p}$
where $p$ runs over the primes. We want to use the inclusion exclusion
principle to find the size of $prim^{+}\left(N\right)$, but we can
only do if the union was over only finitely many primes. For that,
let $P_{M}=\prod_{1}^{M}p_{i}$ be the product of the first $M$ primes,
then
\[
-\frac{1}{N^{2}}\sum_{p\nmid P_{M}}\left|U_{p}\right|\leq\frac{1}{N^{2}}\left(\left|U_{1}\backslash\bigcup_{p}U_{p}\right|-\left|U_{1}\backslash\bigcup_{p\mid P_{M}}U_{p}\right|\right)\leq0.
\]

Before doing the inclusion exclusion, note that 
\[
\frac{1}{N^{2}}\sum_{p\nmid P_{M}}\left|U_{p}\right|\leq\frac{1}{N^{2}}\sum_{p\geq M}\flr{\frac{N}{p}}^{2}\leq\sum_{p\geq M}\frac{1}{p^{2}}.
\]
The series $\sum_{p}\frac{1}{p^{2}}$ converge, so that $\sum_{p\geq M}\frac{1}{p^{2}}\to0$
as $M=M\left(N\right)\to\infty$. In this case we have 
\[
\lim_{N}\frac{1}{N^{2}}\left|prim\left(N\right)\right|=\lim_{N}\frac{1}{N^{2}}\left|U_{1}\backslash\bigcup_{p\mid P_{M}}U_{p}\right|,
\]
if the limits exist, and we can show this using the inclusion exclusion
principle.

Let $\mu$ be the M{\"o}bius function, namely $\mu\left(P\right)=\left(-1\right)^{k}$
if $P$ is a product of $k$ distinct primes and $\mu\left(P\right)=0$
otherwise. Since ${\displaystyle \bigcap_{p\mid P}}U_{p}=U_{P}$ by
definition, we get that $\left|U_{1}\backslash\bigcup_{p\mid P_{M}}U_{p}\right|=\sum_{P\mid P_{M}}\mu\left(P\right)\left|U_{P}\right|$
and $\left|U_{P}\right|=\flr{\frac{N}{P}}^{2}$, so that
\begin{align*}
\left|\left|U_{1}\backslash\bigcup_{p\mid P_{M}}U_{p}\right|-\sum_{P\mid P_{M}}\mu\left(P\right)\left(\frac{N}{P}\right)^{2}\right| & =\left|\sum_{P\mid P_{M}}\mu\left(P\right)\left(\flr{\frac{N}{P}}^{2}-\left(\frac{N}{P}\right)^{2}\right)\right|\leq2^{M+1}N.
\end{align*}
If we choose for example $M\left(N\right)=\frac{\log_{2}\left(N\right)}{2}$,
then $\frac{2^{M+1}N}{N^{2}}=\frac{2}{\sqrt{N}}\to0$, so we are left
with computing the limit for $\sum_{P\mid P_{M}}\mu\left(P\right)\left(\frac{1}{P}\right)^{2}$.
We can now write the last term as
\[
\sum_{P\mid P_{M}}\mu\left(P\right)\left(\frac{1}{P}\right)^{2}=\sum_{P\mid P_{M}}\prod_{p\mid P}\frac{\left(-1\right)}{p^{2}}=\prod_{p\mid P_{M}}\left(1-\frac{1}{p^{2}}\right).
\]
The limit of these products as $M\to\infty$ is also well known. Indeed,
we have that 
\[
\frac{\pi^{2}}{6}=\zeta\left(2\right)=\sum_{1}^{\infty}\frac{1}{n^{2}}=\prod_{p}\sum_{i=0}^{\infty}\frac{1}{p^{2i}}=\prod_{p}\frac{1}{1-1/p^{2}},
\]
so that $\prod_{p\mid P_{M}}\left(1-\frac{1}{p^{2}}\right)\to\zeta\left(2\right)^{-1}$,
which completes the proof.
\end{proof}
\begin{cor}
\label{thm:counting_matrices}For all $R$ big enough we have that
$\left|\SL_{2}\left(\ZZ\right)\cap B_{R}\right|\geq\frac{6}{\pi^{2}}R^{2}$.
\end{cor}

\begin{proof}
This is just a combination of the two previous lemmas.
\end{proof}
Finally, we extend this result to finite index subgroups of $\SL_{2}\left(\ZZ\right)$.
\begin{cor}
Let $\Gamma\leq\SL_{2}\left(\ZZ\right)$ be a finite index subgroup.
Then there exists $C_{\Gamma}>0$ such that for all $R$ big enough
we have that $\left|\Gamma\cap B_{R}\right|\geq C_{\Gamma}R^{2}$.
\end{cor}

\begin{proof}
Let $g_{1}^{-1},...,g_{n}^{-1}$ be coset representatives of $\Gamma$
in $\SL_{2}\left(\RR\right)$. If $g\in\SL_{2}\left(\ZZ\right)\cap B_{R}$,
then $g=g_{i}^{-1}h$ for some $h\in\Gamma$ and $i\leq n$, hence
$\norm h_{\infty}\leq2\norm g_{\infty}\norm{g_{i}}_{\infty}$. Thus,
setting $M=\max\norm{g_{i}}$ we obtain that
\[
CR^{2}\leq\left|\SL_{2}\left(\ZZ\right)\cap B_{R}\right|\leq\sum_{1}^{n}\left|g_{i}^{-1}\left(\Gamma\cap B_{2RM}\right)\right|=n\left|\Gamma\cap B_{2RM}\right|
\]
for all $R$ big enough, hence $\left|\Gamma\cap B_{R}\right|\geq\frac{C}{n4M^{2}}R^{2}$
for all $R$ big enough.
\end{proof}
\newpage{}
\begin{rem}
Behind the curtains of what we did here hides an action of the group
\\
$U=\left\{ \left(\begin{array}{cc}
1 & x\\
0 & 1
\end{array}\right)\;\mid\;x\in\RR\right\} =stab_{\SL_{2}\left(\RR\right)}\left(1,0\right)$ which we saw when we looked for completion from a primitive vector
to a matrix. Note that since $\SL_{2}\left(\RR\right)$ acts transitively
on $\RR^{2}\backslash\left\{ 0\right\} $, we can write it as $\RR^{2}\cong\SL_{2}\left(\RR\right)/U$,
so that the primitive vectors correspond to the orbit $\SL_{2}\left(\ZZ\right)e_{1}\to\SL_{2}\left(\ZZ\right)\cdot Id/U$.
We can reverse the roles of $\SL_{2}\left(\ZZ\right)$ and $U$ and
look on $U$ orbits on the space $\SL_{2}\left(\ZZ\right)\backslash\SL_{2}\left(\RR\right)$.
This duality between left and right orbits let us use results from
one side and translate it to the second. In this case specifically,
the acting group is $U\cong\RR$ is a very simple to work with group,
and it is usually called the horocycle group. This type of orbits
are well known and mostly understood, and this process is used often
to count lattice points. For more details, see \cite{eskin_mixing_1993}.
\end{rem}

\bibliographystyle{plain}
\bibliography{girth}

\end{document}